\documentclass[10pt]{article}

\usepackage[margin=1.12in]{geometry}
\usepackage[T1]{fontenc}
\usepackage[utf8]{inputenc}
\usepackage{amsmath,amssymb,amsthm,mathtools}
\usepackage{authblk}
\usepackage{enumitem}
\usepackage{aliascnt}
\usepackage[hyperfootnotes=false]{hyperref}
\usepackage[nameinlink,noabbrev]{cleveref}
\usepackage{microtype}

\hypersetup{colorlinks=true,linkcolor=blue,citecolor=blue,urlcolor=blue}

\setlist[enumerate]{leftmargin=2.1em,itemsep=0.15em,topsep=0.25em}
\allowdisplaybreaks[2]

\theoremstyle{plain}
\newtheorem{theorem}{Theorem}[section]
\newaliascnt{lemma}{theorem}
\newtheorem{lemma}[lemma]{Lemma}
\aliascntresetthe{lemma}
\newaliascnt{proposition}{theorem}
\newtheorem{proposition}[proposition]{Proposition}
\aliascntresetthe{proposition}
\newaliascnt{corollary}{theorem}

\aliascntresetthe{corollary}
\newaliascnt{conjecture}{theorem}

\aliascntresetthe{conjecture}
\newaliascnt{question}{theorem}
\newtheorem{question}[question]{Question}
\aliascntresetthe{question}
\newaliascnt{claim}{theorem}
\newtheorem{claim}[claim]{Claim}
\aliascntresetthe{claim}
\theoremstyle{remark}
\newaliascnt{remark}{theorem}

\aliascntresetthe{remark}
\newtheorem*{remark*}{Remark}

\crefname{theorem}{Theorem}{Theorems}
\crefname{lemma}{Lemma}{Lemmas}
\crefname{proposition}{Proposition}{Propositions}
\crefname{corollary}{Corollary}{Corollaries}
\crefname{conjecture}{Conjecture}{Conjectures}
\crefname{question}{Question}{Questions}
\crefname{claim}{Claim}{Claims}
\crefname{section}{Section}{Sections}
\crefname{subsection}{Subsection}{Subsections}
\crefname{equation}{Equation}{Equations}
\crefalias{enumi}{proposition}

\DeclareMathOperator{\Aut}{Aut}
\DeclareMathOperator{\Emb}{Emb}
\newcommand{\symdiff}{\mathbin{\triangle}}
\newcommand{\eps}{\varepsilon}
\newcommand{\bmu}{{\boldsymbol{\mu}}}
\newcommand{\bnu}{{\boldsymbol{\nu}}}
\newcommand{\bomega}{{\boldsymbol{\omega}}}

\title{Exact extremal constructions for the inducibility of blowup graphs}
\author{Wanfang~Chen\thanks{Email: \texttt{wanfangchen.math@gmail.com}}
  \quad and\quad Xizhi~Liu\thanks{Email: \texttt{liuxizhi@ustc.edu.cn}}}
\affil{\small School of Mathematical Sciences, University of Science and Technology of China, Hefei, China}
\date{\today}

\makeatletter
\renewcommand{\maketitle}{%
  \par
  \begingroup
    \renewcommand\thefootnote{\@fnsymbol\c@footnote}%
    \def\@makefnmark{\rlap{\@textsuperscript{\normalfont\@thefnmark}}}%
    \long\def\@makefntext##1{\parindent 1em\noindent
      \hb@xt@1.8em{\hss\@textsuperscript{\normalfont\@thefnmark}}##1}%
    \begin{center}
      \let\footnote\thanks
      {\Large\bfseries \@title\par}
      \vspace{0.6em}
      {\normalsize \@author\par}
      \vspace{0.4em}
      {\normalsize \@date\par}
    \end{center}
    \@thanks
  \endgroup
  \setcounter{footnote}{0}%
  \vspace{0.8em}
}
\renewenvironment{abstract}{%
  \begin{center}\bfseries Abstract\end{center}
  \begin{quote}\small\noindent\ignorespaces
}{%
  \end{quote}\vspace{0.6em}
}
\makeatother

\begin{document}

\maketitle

\begin{abstract}
For a finite graph $H$ and a positive integer $h$, the $h$-blowup $H^{(h)}$ of $H$ is the graph obtained by replacing each vertex of $H$ by a set of size $h$ and each edge by a complete bipartite graph between the corresponding sets. We prove that, for every $H$, there exists a constant $h_*(H)$ such that whenever $h\ge h_*(H)$ and $n$ is sufficiently large, every $n$-vertex graph maximizing the number of induced copies of $H^{(h)}$ is a blowup of $H$. This refines the asymptotic result of Hatami, Hirst and Norine and settles the question posed by Bollob\'as, Egawa, Harris and Jin in 1995.
\end{abstract}

\section{Introduction}
\label{sec:introduction}

A fundamental problem in extremal graph theory is to determine the maximum number of induced copies of a given graph $F$ among all $n$-vertex graphs. Formally, given two graphs $F$ and $G$, let $I(F,G)$ denote the number of induced copies of $F$ in $G$, that is, the number of subsets $S\subseteq V(G)$ of size $v(F)$ such that the induced subgraph $G[S]$ is isomorphic to $F$. Here $v(F)$ denotes the number of vertices of $F$. For every positive integer $n$, put
\[
        I(F,n)\coloneqq \max\{I(F,G):v(G)=n\}.
\]
The \emph{inducibility} of $F$ is then defined as
\[
        i(F)\coloneqq\lim_{n\to\infty}{I(F,n)}/{\tbinom{n}{v(F)}}.
\]
The limit exists by the standard averaging argument.

A systematic study of the inducibility problem was initiated by Pippenger--Golumbic~\cite{PippengerGolumbic}, who established several general properties of inducibility and determined the inducibility of complete bipartite graphs with part sizes differing by at most one. Determining $I(F,n)$, or even $i(F)$, is difficult in general. For example, the inducibility of the path on four vertices remains open; see, for instance, Exoo~\cite{Exoo} and Even-Zohar--Linial~\cite{EvenZoharLinial}. For small graphs, bounds on the inducibility of graphs on four, five, and six vertices have been obtained in several works, including Exoo~\cite{Exoo}, Hirst~\cite{HirstFourVertices}, Even-Zohar--Linial~\cite{EvenZoharLinial}, Pikhurko--Slia{\v c}an--Tyros~\cite{Sliacan19}, and Bodn\'ar et al.~\cite{BGLLPS26}, some of which employ the computer-assisted flag algebra machinery of Razborov~\cite{Raz07}.

A particularly interesting result of Balogh--Hu--Lidick\'y--Pfender~\cite{BaloghHuLidickyPfender} is the determination of the inducibility of the $5$-cycle $C_5$, confirming a special case of an old conjecture of Pippenger--Golumbic~\cite{PippengerGolumbic} on the inducibility of cycles. This conjecture remains open for longer cycles, and improved upper bounds were obtained by Hefetz--Tyomkyn~\cite{HefetzTyomkyn} and by Kr\'al'--Norin--Volec~\cite{KralNorinVolec}. For general graphs, results of Yuster~\cite{YusterAlmostAll} and Fox--Huang--Lee~\cite{FoxHuangLee} imply that for almost all graphs $F$, it holds that $i(F)=\frac{v(F)!}{v(F)^{v(F)}-v(F)}$, where the lower-bound construction arises from nested blowups of $F$ itself. Extending the results of~\cite{MMNT19,KST19,FS20} on the edge-statistic conjecture of Alon--Hefetz--Krivelevich--Tyomkyn~\cite{AHKT20}, Ueltzen~\cite{Uel24} recently characterized graphs whose inducibility is bounded away from zero in the large-order regime.

In this work, we focus on the blowup of graphs. If $H$ is a graph and $\mathbf k\in\mathbb N^{V(H)}$, the blowup $H(\mathbf k)$ is obtained by replacing each vertex $x\in V(H)$ by an independent set of size $\mathbf k(x)$, and by joining two replacement classes completely exactly when the corresponding vertices of $H$ are adjacent. If all coordinates of $\mathbf k$ are equal to $h$, we write $H^{(h)}$ and call it the \emph{balanced $h$-blowup} of $H$.

Blowups are natural candidates for extremal constructions in many inducibility problems.
Bollob\'as, Egawa, Harris and Jin~\cite{BollobasEgawaHarrisJin} proved the complete-graph case: if $t>1+\log r$, then, for all sufficiently large $n$, the maximum number of induced copies of $K_r^{(t)}$ in an $n$-vertex graph is attained by the Tur\'an graph $T_r(n)$, and this lower bound on $t$ is necessary, since for smaller $t$ the extremal construction is no longer a blowup of $K_r$. They then asked which base graphs admit this exact blowup conclusion:

\begin{question}[\cite{BollobasEgawaHarrisJin}]
For which graphs $H$ does there exist a constant $h$ such that, for all sufficiently large $n$, every $n$-vertex graph with the maximum number of induced copies of $H^{(h)}$ is a blowup of $H$?
\end{question}

Hatami, Hirst and Norine~\cite{HatamiHirstNorine} proved the corresponding asymptotic theorem for every graph $H$. We briefly recall the language in which their result is stated. A weighted graph $X^\bmu$ is a finite graph $X$ together with a probability measure $\bmu$ on $V(X)$; ordinary finite graphs are viewed with the uniform measure. A map $\phi:V(F)\to V(X)$ is a strong homomorphism if it preserves both edges and non-edges. Thus $s(F;X^\bmu)$ is the probability that a random map from $V(F)$ to $V(X)$, chosen independently according to $\bmu$, realizes the same adjacency pattern as $F$. The use of arbitrary maps rather than injective maps is natural in the asymptotic setting, since collisions have negligible probability in large finite graphs. Finally, $H^\bnu$ denotes the graph $H$ with vertex weights $\bnu$, and $M(V(H))$ is the set of full-support probability measures on $V(H)$; these weights encode the limiting part sizes of blowups of $H$.
\begin{theorem}[\cite{HatamiHirstNorine}]
Let $H$ be a graph. There exists a positive integer $h_0$ such that for every $h\ge h_0$, there exists a distribution $\bnu\in M(V(H))$ with
\[
        \sup_{X^\bmu} s\bigl(H^{(h)};X^\bmu\bigr)
        = s\bigl(H^{(h)};H^\bnu\bigr),
\]
where the supremum is over all finite weighted graphs $X^\bmu$.
\end{theorem}
Equivalently, the inducibility of $H^{(h)}$ is asymptotically achieved by blowups of $H$. In~\cite[Question~7.1]{HatamiHirstNorine}, Hatami, Hirst and Norine returned to the exact finite problem above, asking for the corresponding exact blowup conclusion for induced copies, and indicated how their local lemmas and variational argument should be adapted in order to treat induced embeddings.

The purpose of this note is to supply these finite details. Our main result is the following.

\begin{theorem}
\label{thm:main}
Let $H$ be a non-empty finite graph. There exists an integer $h_*=h_*(H)$ such that for every $h\ge h_*$ there exists $n_0=n_0(H,h)$ with the following property. If $n\ge n_0$ and $G$ is an $n$-vertex graph maximizing the number of induced copies of $H^{(h)}$, then $G$ is a blowup of $H$.
\end{theorem}

Thus the question of Bollob\'as, Egawa, Harris and Jin has an affirmative answer for every graph $H$. The theorem is stated for induced copies; in the proof we count labelled induced embeddings, which differ by the fixed factor $|\Aut(H^{(h)})|$ and hence have the same extremal graphs.

The proof essentially follows the exact-stability framework of~\cite[Theorem~5.8]{Sliacan19}, with the weighted argument of Hatami, Hirst and Norine supplying the asymptotic input. The finite exact conclusion requires several additional steps. The first is a finite substitute for derivative optimality. With $F\coloneqq H^{(h)}$, if $G$ maximizes the number $N_F(G)$ of labelled induced embeddings, then replacing one vertex of $G$ by a clone of another vertex cannot increase $N_F(G)$. This gives, in \cref{lem:finite-local-optimality}, a finite local optimality condition with an error of order $O_F(1/n)$.

The second step turns the weighted input of Hatami, Hirst and Norine into a finite stability statement. Their direct results are weighted statements: a continuity estimate, their asymptotic maximum theorem, their global stability lemma, and a rooted stability lemma, recorded below in \cref{lem:weighted-continuity,prop:weighted-input,lem:weighted-rooted}. Writing $R\coloneqq\widetilde H$ for the twin-free quotient of $H$, so that $H=R(\mathbf k)$, we combine \cref{lem:finite-local-optimality} with quotient and rooted adaptations of their weighted machinery to prove \cref{prop:finite-stability}: for every small fixed $\eta>0$, provided first $h$ and then $n$ are large enough, every exact extremal graph $G$ is close to a blowup $B$ of $R$ on the same vertex set, with all parts of linear size and with $\Delta(G\symdiff B)\le \eta n$.

The third step is the exact cleaning argument. Put $J\coloneqq G\symdiff B$. Each wrong pair in $J$ destroys many type-respecting embeddings present in $B$, while \cref{lem:many-forced-errors} implies that any embedding created by the error graph must contain a large star in $J$. Since $\Delta(J)\le\eta n$, \cref{lem:cleaning} shows that the lost embeddings dominate the created embeddings once $\eta$ is chosen sufficiently small and $h$ sufficiently large. Hence a non-empty error graph would contradict extremality, so $G=B$. The refinement from the twin-free quotient back to $H$ is then immediate from the linear lower bound on the quotient parts.

\medskip

\cref{sec:notation} fixes notation, records the weighted results of Hatami, Hirst and Norine, proves the quotient form of the rooted result, and proves the finite local-optimality lemma. \cref{sec:proof} proves \cref{thm:main}: \cref{sec:finite-stability} derives the finite local blowup structure from these results, \cref{sec:cleaning} proves the cleaning lemma, and \cref{sec:main-proof} completes the proof.

\section{Preliminaries}
\label{sec:notation}

For a graph $G$, write $d_G(x)$ for the degree of $x$ in $G$ and $\Delta(G)$ for the maximum degree of $G$.
We identify a graph with its edge set; thus $|G|$ denotes the number of edges of $G$.
We write $A_G(x,y)$ for the $(x,y)$-entry of the adjacency matrix of $G$, with zero diagonal. Thus $A_G(x,y)=1$ if $xy\in G$ and $A_G(x,y)=0$ otherwise.

For graphs $F$ and $G$, let $\Emb(F,G)$ be the set of injective maps $\phi:V(F)\to V(G)$ such that, for all $x,y\in V(F)$, $xy\in F$ if and only if $\phi(x)\phi(y)\in G$.
Put $N_F(G)\coloneqq|\Emb(F,G)|$. If $m\coloneqq |V(F)|$ and $n\coloneqq |V(G)|\ge m$, set
\[
        p_F(G)\coloneqq {N_F(G)}/{(n)_m},
\]
where $(n)_m\coloneqq n(n-1)\cdots(n-m+1)$. Thus $p_F(G)$ is the probability that a uniform random injective map $V(F)\hookrightarrow V(G)$ is an induced embedding.
All embedding counts are labelled. Thus
\[
        N_F(G)=|\Aut(F)|\,I(F,G).
\]

A map $\phi:V(F)\to V(G)$ is a \emph{strong homomorphism} if it preserves both edges and non-edges:
\[
        A_F(x,y)=A_G(\phi(x),\phi(y))
        \qquad\text{for all }x,y\in V(F).
\]
Thus an induced embedding is precisely an injective strong homomorphism. The only difference is that a strong homomorphism may identify distinct vertices, whereas an induced embedding may not.

Two vertices in $H$ are called \emph{twins} if they have the same neighbourhood (and thus are not adjacent). This is an equivalence relation on $V(H)$, and its equivalence classes are independent sets. The \emph{twin-free quotient} $\widetilde H$ is the graph whose vertices are the twin classes of $H$, with two classes adjacent when, equivalently, every pair of vertices between them is adjacent.
Writing $R\coloneqq\widetilde H$, for each $x\in V(R)$, let $\mathbf k(x)$ be the size of the corresponding twin class in $H$. We write
\[
        \|\mathbf k\|_1\coloneqq\sum_{x\in V(R)}\mathbf k(x),
        \qquad
        \|\mathbf k\|_\infty\coloneqq\max_{x\in V(R)}\mathbf k(x).
\]
Then $H$ is the blowup $R(\mathbf k)$, and hence $H^{(h)}=R(h\mathbf k)$.
The natural projection $\pi:R(h\mathbf k)\to R$, which sends each blown-up vertex to the quotient vertex whose part contains it, is called the \emph{type map}.

We shall also use the following elementary fact several times.

\begin{lemma}
\label{lem:twin-free-auto}
If $R$ is twin-free, then every strong homomorphism $R\to R$ is an automorphism.
\end{lemma}
\begin{proof}
If $\phi(x)=\phi(y)$, then for every $z\in V(R)$, $A_R(x,z)=A_R(\phi(x),\phi(z))=A_R(\phi(y),\phi(z))=A_R(y,z)$. Thus $x$ and $y$ are twins, and so $x=y$. Hence $\phi$ is injective and therefore bijective. Since $\phi$ preserves adjacency and non-adjacency, it is an automorphism.
\end{proof}

We recall the weighted notions of Hatami, Hirst and Norine~\cite[Section~2.1]{HatamiHirstNorine} that are used below.
For a finite set $S$, put
\[
        \mathcal P(S)\coloneqq
        \Big\{ \bmu:S\to[0,1]\;:\;\sum_{x\in S}\bmu(x)=1 \Big\},
        \quad\text{and}\quad
        M(S)\coloneqq\{\bmu\in\mathcal P(S):\bmu(x)>0\text{ for every }x\in S\}.
\]
Thus, $M(S)$ denotes the full-support part of the simplex $\mathcal P(S)$; the closed simplex $\mathcal P(S)$ will only be needed when taking closest representatives.

Recall that a \emph{weighted graph} $X^\bmu$ is a finite graph $X$ together with a probability measure $\bmu\in\mathcal P(V(X))$. If no measure is displayed, we use the uniform measure and write it as $u_X$.
Throughout, when discussing equivalence and edit distance, weighted graphs are identified with their reduced representatives obtained by deleting all zero-weight vertices.
For a weighted graph $X^\bmu$, let $s(F;X^\bmu)$ be the probability that a random map $V(F)\to V(X)$, chosen independently according to $\bmu$, is a strong homomorphism. We also use a \emph{rooted density}. For $x\in V(X)$, let $T_x(F;X^\bmu)$ be the probability that the following experiment gives a strong homomorphism: choose a uniform root $y\in V(F)$, force $y$ to be mapped to $x$, and map all other vertices independently according to $\bmu$. Then
\begin{equation}
\label{eq:root-average}
        s(F;X^\bmu)=\sum_{x\in V(X)}\bmu(x)T_x(F;X^\bmu).
\end{equation}

Two weighted graphs $X^\bmu$ and $Y^\bnu$ are \emph{equivalent}, written $X^\bmu\sim Y^\bnu$, if there is an isomorphism $\theta:\widetilde X\to\widetilde Y$ between their twin-free quotients such that, for every quotient vertex $x\in V(\widetilde X)$, the total $\bmu$-weight of the twin class of $x$ equals the total $\bnu$-weight of the twin class of $\theta(x)$. They are \emph{commensurable} if they have the same vertex set and the same probability measure.
A \emph{weighted blowup} of $X^\bmu$ means any weighted graph equivalent to $X^\bmu$.

The \emph{weighted edit distance} between two weighted graphs $X^\bmu$ and $Y^\bnu$ is
\[
        \mathrm{edit}(X^\bmu,Y^\bnu)\coloneqq
        \min \sum_{x,y\in V(W_1)} \bomega(x)\bomega(y)
        \left|A_{W_1}(x,y)-A_{W_2}(x,y)\right|,
\]
where the minimum is over all commensurable weighted graphs $W_1^\bomega$ and $W_2^\bomega$ such that $W_1^\bomega\sim X^\bmu$ and $W_2^\bomega\sim Y^\bnu$.
The sum is over ordered pairs; this normalization is immaterial for us. Zero-weight vertices may still be introduced temporarily during refinements, but they are then ignored when passing back to the reduced representative.

\begin{lemma}[{\cite[Equation~(2.1)]{HatamiHirstNorine}}]
\label{lem:weighted-continuity}
For every finite graph $F$ and weighted graphs $X^\bmu,Y^\bnu$,
\[
        |s(F;X^\bmu)-s(F;Y^\bnu)|
        \le |V(F)|^2 \mathrm{edit}(X^\bmu,Y^\bnu).
\]
\end{lemma}

The next proposition records two weighted consequences of the work of Hatami, Hirst and Norine that will be used below: their maximum theorem and their global stability lemma. It is not yet an exact finite statement; the finite adaptation is carried out in \cref{sec:finite-stability}.

\begin{proposition}[{\cite[Theorem~3.2 and Lemma~5.1]{HatamiHirstNorine}}]
\label{prop:weighted-input}
Let $H=R(\mathbf k)$, where $R$ is twin-free, and put $F_h\coloneqq H^{(h)}=R(h\mathbf k)$.
\begin{enumerate}[label=\textup{(\alph*)},ref=\theproposition\textup{(\alph*)}]
\item\label{prop:weighted-input:maximum} There is $h_0$ such that, for every $h\ge h_0$,
\[
        \sup_{X^\bmu} s(F_h;X^\bmu)=\max_{\bnu\in M(V(R))} s(F_h;R^\bnu),
\]
where the supremum is over all finite weighted graphs $X^\bmu$.
\item\label{prop:weighted-input:global} For every $\delta>0$ there is $h_1=h_1(\delta)$ such that, whenever $h\ge h_1$ and $s(F_h;X^\bmu)\ge \frac12 s(F_h;H)$, we have $\mathrm{edit}(X^\bmu,H)\le\delta$.
\end{enumerate}
\end{proposition}

Here \cref{prop:weighted-input:global} is Hatami, Hirst and Norine's global stability lemma in its relative-threshold form; the comparison value $\frac12 s(F_h;H)$ is the threshold appearing in that stability statement.

We shall need one more result of Hatami, Hirst and Norine, namely their rooted stability lemma~\cite[Lemma~5.3]{HatamiHirstNorine}. Since their Lemma 5.3 is stated for weighted blowups of the original graph $H$ and uses a single parameter to control regularity, the multiplicative loss, and the minimum part weight, we record explicitly the quotient form needed here. The proof is just the parameter conversion from their statement to the two-parameter form below.

\begin{lemma}[{\cite[Lemma~5.3]{HatamiHirstNorine}}]
\label{lem:weighted-rooted}
Let $H=R(\mathbf k)$, where $R$ is twin-free, and put $F_h\coloneqq H^{(h)}=R(h\mathbf k)$. For every $\rho,\lambda>0$ there are $h_2=h_2(\rho,\lambda)$ and $\delta>0$ such that the following holds for $h\ge h_2$. Suppose $X^\bmu$ and $B^\bmu$ are commensurable weighted graphs, $B^\bmu$ is equivalent to a weighted blowup $R^\bnu$, $\min_{y\in V(R)}\bnu(y)\ge\lambda$, and $\mathrm{edit}(X^\bmu,B^\bmu)\le\delta$. If a positive-weight vertex $x\in V(X)$ has no row of $B$ within $\rho$ in $\bmu$-measure, that is, if
\[
        \bmu\{z:A_X(x,z)\ne A_B(y,z)\}>\rho
        \qquad\text{for every }y\in V(B),
\]
then
\[
        T_x(F_h;X^\bmu)\le \rho\,T_x(F_h;B^\bmu).
\]
\end{lemma}

\begin{proof}
Choose $\eps_0$ such that $0<\eps_0\le \min\{\rho,\lambda/\|\mathbf k\|_\infty\}$.
Split the quotient weight $\bnu(y)$ among the $\mathbf k(y)$ vertices of the twin class of $H$ lying over $y$; for instance, put
\[
        \widehat{\bnu}(x)\coloneqq \frac{\bnu(y)}{\mathbf k(y)}
        \qquad\text{whenever }x\in V(H)\text{ lies over }y\in V(R).
\]
Then every coordinate of $\widehat{\bnu}$ is at least $\lambda/\|\mathbf k\|_\infty$, and hence at least $\eps_0$. Thus the minimum vertex weight condition in Hatami, Hirst and Norine's lemma is satisfied. Since vertices lying over the same quotient vertex are twins, the weighted graph $B^\bmu$ is also equivalent to the weighted blowup $H^{\widehat{\bnu}}$.

Before applying their lemma, delete all zero-weight vertices from the common vertex set of the commensurable pair. Since $x$ has positive weight, it remains after this deletion. The deletion does not change the edit distance, the equivalence class of $B^\bmu$, the rooted densities at $x$, or any of the row distances from $x$ to the remaining rows of $B$.

Hatami, Hirst and Norine's Lemma 5.3 says, in their notation, that if a vertex is not $\eps_0$-regular with respect to a commensurable pair consisting of a weighted graph and a nearby weighted blowup of $H^{\widehat{\bnu}}$, then its rooted density is smaller by a factor at most $\eps_0$. In this reduced commensurable representative, their row distance $d_1(x,y)$ is exactly
\[
        d_1(x,y)=
        \sum_{z\in V(B)}\bmu(z)\left|A_X(x,z)-A_B(y,z)\right|
        =\bmu\{z:A_X(x,z)\ne A_B(y,z)\}.
\]
Thus their $\eps_0$-regularity condition is precisely the existence of some row of $B$ within $\eps_0$ in the above sense. Apply that lemma to the graph $H$ with parameter $\eps_0$, obtaining $h_2$ and $\delta$. If $x$ has no row of $B$ within $\rho$ in $\bmu$-measure, then, because $\eps_0\le\rho$, it has no row of $B$ within $\eps_0$ either. In their terminology, $x$ is therefore not $\eps_0$-regular with respect to the reduced commensurable pair $(X,B,\bmu)$.

In Hatami, Hirst and Norine's partially labelled notation, $\partial H^{(h)}$ denotes the sum of the one-vertex partial labellings of $H^{(h)}$. Thus their rooted density is the sum of the contributions obtained by fixing each possible root at $x$, whereas our $T_x$ averages over the root. Equivalently, for $W=X,B$,
\[
        s(\partial H^{(h)},1\mapsto x;W^\bmu)
        = |V(F_h)|\,T_x(F_h;W^\bmu).
\]
Their lemma gives
\[
        s(\partial H^{(h)},1\mapsto x;X^\bmu)
        \le \eps_0\,s(\partial H^{(h)},1\mapsto x;B^\bmu).
\]
Using the displayed identity for $W=X$ and $W=B$, and using $\eps_0\le\rho$, we obtain
\[
        T_x(F_h;X^\bmu)\le \rho\,T_x(F_h;B^\bmu),
\]
as required.
\end{proof}

The following elementary lemma replaces the derivative optimality condition in the weighted argument. It follows by comparing an extremal graph with the graph obtained after cloning one vertex by another.

\begin{lemma}
\label{lem:finite-local-optimality}
Let $F$ be a fixed graph, with $m\coloneqq |V(F)|$. There is a constant $C_F$ such that for every $n\ge2m$, if $G$ is an $n$-vertex graph maximizing $N_F(G)$, then
\[
        T_x(F;G^{u_G})\ge p_F(G)-\frac{C_F}{n}
        \qquad\text{for all }x\in V(G),
\]
where $u_G$ is the uniform probability measure on $V(G)$. One may take $C_F=4m^2$.
\end{lemma}

\begin{proof}
Let $\mathcal E\coloneqq\Emb(F,G)$. For $x\in V(G)$, let $Q_x$ be the number of pairs $(y,\phi)\in V(F)\times\mathcal E$ with $\phi(y)=x$, and put
\[
        q_x\coloneqq\frac{Q_x}{m(n-1)_{m-1}}.
\]
Since the maps in $\mathcal E$ are injective, $Q_x$ is equivalently the number of embeddings whose image contains $x$, counted with the unique preimage of $x$.
Then $n^{-1}\sum_x q_x=p_F(G)$. Fix $x\in V(G)$. If $q_x\ge p_F(G)$, there is nothing to prove at this stage. Otherwise choose $x'\ne x$ with $q_{x'}\ge p_F(G)$. Form $G'$ from $G$ by replacing $x$ with a non-adjacent clone of $x'$: delete all edges incident with $x$, join $x$ to $z\notin\{x,x'\}$ exactly when $x'$ is joined to $z$ in $G$, and leave $xx'$ as a non-edge. Embeddings not using $x$ are unchanged, and embeddings using $x$ but not $x'$ in $G'$ correspond to embeddings using $x'$ but not $x$ in $G$. The only maps not accounted for are those using both specified vertices, at most $m(m-1)(n-2)_{m-2}$. Since $G$ is extremal,
\[
        Q_x\ge Q_{x'}-m(m-1)(n-2)_{m-2},
\]
and hence $q_x\ge p_F(G)-2m/n$ for $n\ge2m$. Thus, in all cases,
\[
        q_x\ge p_F(G)-\frac{2m}{n}.
\]

Finally, the rooted map defining $T_x(F;G^{u_G})$ is non-injective only if some non-root vertex maps to $x$ or two non-root vertices collide. This has probability at most $m^2/n$. Therefore
\[
        |T_x(F;G^{u_G})-q_x|\le \frac{m^2}{n},
\]
which proves the lemma, after enlarging the constant to $4m^2$.
\end{proof}

We shall also need to know that a closest weighted quotient blowup has no vanishing part. This is what later lets us round weighted parts to linear-sized finite parts.

\begin{lemma}
\label{lem:boundary}
Let $H=R(\mathbf k)$, with $R$ twin-free. There are constants $\lambda_R,\delta_R>0$, depending only on $H$, with the following property. If $\mathrm{edit}(X^\bmu,H)\le\delta_R$ and $\bnu\in\mathcal P(V(R))$ minimizes $\mathrm{edit}(X^\bmu,R^\bnu)$, then $\bnu(x)\ge\lambda_R$ for every $x\in V(R)$.
\end{lemma}

\begin{proof}
Such minimizers exist: the simplex $\mathcal P(V(R))$ is compact, and $\bnu\mapsto\mathrm{edit}(X^\bmu,R^\bnu)$ is continuous. As above, weighted graphs that differ only by zero-weight vertices represent the same equivalence class, so the triangle inequality for $\mathrm{edit}$ applies on these equivalence classes.

Let $\bmu_{\mathbf k}(x)\coloneqq\mathbf k(x)/\|\mathbf k\|_1$. Since $H$ is equivalent to $R^{\bmu_{\mathbf k}}$, the minimizing property gives
\[
        \mathrm{edit}(X^\bmu,R^\bnu)\le \mathrm{edit}(X^\bmu,R^{\bmu_{\mathbf k}})=\mathrm{edit}(X^\bmu,H).
\]
Using that $\mathrm{edit}$ is a metric on equivalence classes of weighted graphs, we obtain
\[
        \mathrm{edit}(R^\bnu,R^{\bmu_{\mathbf k}})
        \le \mathrm{edit}(R^\bnu,X^\bmu)+\mathrm{edit}(X^\bmu,R^{\bmu_{\mathbf k}})
        \le 2\delta_R.
\]
By \cref{lem:weighted-continuity}, applied to $F=R$, the number $s(R;R^\bnu)$ is close to $s(R;R^{\bmu_{\mathbf k}})$ when $\delta_R$ is small. Since $R$ is twin-free, every strong homomorphism $R\to R$ is an automorphism by \cref{lem:twin-free-auto}. Hence
\[
        s(R;R^\bnu)=|\Aut(R)|\prod_{x\in V(R)}\bnu(x).
\]
The value at $\bmu_{\mathbf k}$ is positive. Taking $\delta_R$ small enough forces $\prod_x\bnu(x)\ge c_H>0$. Since all coordinates are at most one, each coordinate is at least $c_H$. We may take $\lambda_R\coloneqq c_H$.
\end{proof}

\section{Proof of the main theorem}
\label{sec:proof}

We now assemble the ingredients from the previous section. The proof has three steps. First, finite local optimality and the weighted stability results show that every exact extremal graph is close, in maximum error degree, to a blowup $B$ of the twin-free quotient $R$. Second, a cleaning lemma shows that the error graph $J\coloneqq G\symdiff B$ must be empty once $h$ is large. Finally, the resulting quotient blowup is refined back to a blowup of the original graph $H$.

\subsection{Stability in the finite setting}
\label{sec:finite-stability}

The first step is to turn the weighted stability results into a statement about an exact finite extremal graph. \Cref{prop:weighted-input:global} gives closeness to some weighted blowup of the quotient $R$, while \cref{lem:weighted-rooted} and \cref{lem:finite-local-optimality} force each vertex to have a row close to one of the quotient rows. This yields a blowup of $R$ on the original vertex set, with all parts linear and with only small maximum-degree error.

\begin{proposition}
\label{prop:finite-stability}
Let $H=R(\mathbf k)$, where $R$ is twin-free. There is $\lambda_0=\lambda_0(H)>0$ such that for every $\eta>0$ there exists $h_{\ref{prop:finite-stability}}=h_{\ref{prop:finite-stability}}(H,\eta)$ with the following property. For every $h\ge h_{\ref{prop:finite-stability}}$ there is $n_{\ref{prop:finite-stability}}=n_{\ref{prop:finite-stability}}(H,h,\eta)$ such that, if $n\ge n_{\ref{prop:finite-stability}}$ and $G$ is an $n$-vertex graph maximizing $N_{H^{(h)}}(G)$, then there is a blowup $B$ of $R$ on $V(G)$ such that
\[
        |B_v|\ge\lambda_0 n\quad\text{for every }v\in V(R),
        \quad\text{and}\quad
        \Delta(G\symdiff B)\le\eta n.
\]
\end{proposition}

\begin{proof}
Fix $\eta>0$.
Let $\lambda_R,\delta_R$ be given by \cref{lem:boundary}, and set $\lambda_0\coloneqq\lambda_R/2$. Choose $\rho$ such that $0<\rho<\min\{\eta/20,\lambda_R/4\}$. Let $\delta_{\mathrm{root}}>0$ and $h_{\mathrm{root}}$ be supplied by \cref{lem:weighted-rooted} for this $\rho$ and for $\lambda=\lambda_R$, and put
\[
        \delta_0\coloneqq\min\{\delta_R,~\delta_{\mathrm{root}}\}.
\]
Take $h$ large enough that $h\ge h_{\mathrm{root}}$, \cref{prop:weighted-input:maximum} applies, and \cref{prop:weighted-input:global} applies with error parameter $\delta_0$. Put $F\coloneqq H^{(h)}$, $m\coloneqq |V(F)|$, and
\[
        M_h\coloneqq\sup_{X^\bmu} s(F;X^\bmu)=\max_{\bnu\in M(V(R))} s(F;R^\bnu).
\]
Throughout this proof, all $o_n(1)$ terms are taken with $H,h,\eta$ fixed.

Let $G$ be an $n$-vertex extremal graph, that is, $N_F(G)=\max\{N_F(G'):|V(G')|=n\}$.

\begin{claim}
\label{claim:extremal-density-lower-bound}
We have $p_F(G)\ge M_h-o_n(1)$.
\end{claim}

\begin{proof}
Choose $\bnu\in M(V(R))$ with $s(F;R^\bnu)=M_h$. Choose integers $n_x$, $x\in V(R)$, such that $\sum_x n_x=n$ and $|n_x-n\bnu(x)|\le1$ for every $x$, and let $\hat G_n$ be the blowup of $R$ with parts of sizes $n_x$. Then $p_F(\hat G_n)=M_h-o_n(1)$. Indeed, compare $p_F(\hat G_n)$ with the probability that an independent uniform map $V(F)\to V(\hat G_n)$ is a strong homomorphism. Since $F$ is fixed once $h$ is fixed, such a map is non-injective with probability $O_h(1/n)$, so conditioning on injectivity changes the probability by $O_h(1/n)$. Also, the empirical part weights $n_x/n$ differ from $\bnu(x)$ by $O_h(1/n)$ for every $x\in V(R)$, and hence the strong-homomorphism probability for $\hat G_n$ differs from $s(F;R^\bnu)=M_h$ by $O_h(1/n)$. Since $G$ is extremal, $p_F(G)\ge p_F(\hat G_n)$, and the claim follows.
\end{proof}

Since every injective embedding is a strong homomorphism,
\[
        s(F;G^{u_G})\ge \frac{(n)_m}{n^m}p_F(G)=M_h-o_n(1).
\]
\cref{lem:finite-local-optimality} gives, for every $x\in V(G)$,
\begin{equation}
\label{eq:all-vertices-rooted}
        T_x(F;G^{u_G})\ge M_h-o_n(1).
\end{equation}
Since $M_h\ge s(F;H)>0$ is fixed once $h$ is fixed, \cref{prop:weighted-input:global} implies, for all sufficiently large $n$, that $\mathrm{edit}(G^{u_G},H)\le\delta_0$.

Choose $\bnu\in\mathcal P(V(R))$ minimizing $\mathrm{edit}(G^{u_G},R^\bnu)$. By \cref{lem:boundary}, $\bnu(x)\ge\lambda_R$ for every $x\in V(R)$. Since $H$ is equivalent to $R^{\bmu_{\mathbf k}}$, where $\bmu_{\mathbf k}(x)=\mathbf k(x)/\|\mathbf k\|_1$, the minimizing choice of $\bnu$ gives
\[
        \mathrm{edit}(G^{u_G},R^\bnu)\le \mathrm{edit}(G^{u_G},R^{\bmu_{\mathbf k}})
        =\mathrm{edit}(G^{u_G},H)\le\delta_0\le\delta_{\mathrm{root}}.
\]
Take commensurable representatives $X^\bomega$ and $Y^\bomega$ such that $X^\bomega\sim G^{u_G}$ and $Y^\bomega\sim R^\bnu$, and such that this pair
realizes the minimum. Then $\mathrm{edit}(X^\bomega,Y^\bomega)\le\delta_{\mathrm{root}}$.

\begin{claim}
\label{claim:refinement-representatives}
The representatives $X^\bomega$ and $Y^\bomega$ may be chosen so that the following properties hold.
\begin{enumerate}[label=\textup{(\roman*)}]
\item The common vertex set admits a partition satisfying $V(X)=V(Y)=\bigsqcup_{x\in V(G)}S_x$ and $\bomega(S_x)=1/n$.
\item The graph $X$ is the usual independent twin-splitting of $G$, that is, there are no edges inside any $S_x$, and $X$ is constant on $S_x\times S_y$, with value $A_G(x,y)$, for $x\ne y$.
\item For every $x\in V(G)$ and every $z\in S_x$, we have $T_z(F;X^\bomega)=T_x(F;G^{u_G})$.
\end{enumerate}
\end{claim}

\begin{proof}
Start with any minimizing commensurable pair. Delete zero-weight vertices from the common vertex set. Since $X^\bomega\sim G^{u_G}$, each twin class $C$ of $G$ corresponds to a twin class $U_C$ of $X$, and $\bomega(U_C)=|C|/n$.
For each $u\in U_C$ and each $x\in C$, replace $u$ by a copy $(u,x)$ of weight $\bomega(u)/|C|$, and put
\[
        S_x\coloneqq\{(u,x):u\in U_C\}.
\]
Then $\bomega(S_x)=1/n$ for every $x\in V(G)$. In $X$, give each copy the same row as $u$, with copies of a single vertex made non-adjacent. Equivalently, the refined $X$ is the independent twin-splitting of $G$ described above. Refine $Y$ on the same copied vertex set, with the same weights, by giving each copy of $u$ the same row as $u$ had in $Y$.

This simultaneous splitting preserves commensurability. It also preserves the weighted edit sum: every original ordered pair $(u,v)$ is replaced by copied ordered pairs whose total weight is $\bomega(u)\bomega(v)$, and the adjacency discrepancy is the same on all those copied pairs. Since splitting vertices into twins does not change the reduced equivalence class, the refined graphs are still equivalent to $G^{u_G}$ and $R^\bnu$, respectively. Hence the refined pair realizes the same minimum.

It remains only to check the rooted-density statement. If $z\in S_x$ and the remaining vertices are sampled from $X^\bomega$, the induced distribution on the labels $y\in V(G)$ is uniform because $\bomega(S_y)=1/n$ for every $y$. The adjacency pattern among the sampled points, and between them and the root $z$, is exactly the adjacency pattern in $G$ among the corresponding labels. Therefore the rooted experiment defining $T_z(F;X^\bomega)$ is the same as the rooted experiment defining $T_x(F;G^{u_G})$.
\end{proof}

Use representatives satisfying \cref{claim:refinement-representatives}, and let $\tau:V(Y)\to V(R)$ be the type map of the weighted blowup $Y$.

\begin{claim}
\label{claim:row-close-to-quotient-row}
For every $z\in V(X)$, there exists $z'\in V(Y)$ such that
\[
        \bomega\{w:A_X(z,w)\ne A_Y(z',w)\}\le\rho.
\]
Moreover, each $z\in V(X)$ is close to its own type row in the minimizing representative:
\begin{equation}
\label{eq:own-row-close}
        \bomega\{w:A_X(z,w)\ne A_Y(z,w)\}\le\rho
        \qquad\text{for all }z\in V(X).
\end{equation}
\end{claim}

\begin{proof}
We first prove the first assertion. Suppose, for a contradiction, that it fails. Then there is a vertex $z\in V(X)$ such that
\[
        \bomega\{w:A_X(z,w)\ne A_Y(z',w)\}>\rho
        \qquad\text{for every }z'\in V(Y).
\]
Thus $z$ satisfies the hypothesis of \cref{lem:weighted-rooted}, applied to the commensurable weighted graphs $X^\bomega$ and $Y^\bomega$. Hence
\[
        T_z(F;X^\bomega)\le \rho\,T_z(F;Y^\bomega).
\]
If $z\in S_x$, then $T_z(F;X^\bomega)=T_x(F;G^{u_G})$, so the left-hand side is at least $M_h-o_n(1)$ by
\eqref{eq:all-vertices-rooted}. On the other hand, $T_z(F;Y^\bomega)\le M_h/\lambda_R$: indeed, all vertices of $Y$ with the same $R$-type $\tau(z)$ have the same rooted density, and their total weight is at least $\lambda_R$, while their contribution to \eqref{eq:root-average} is at most $s(F;Y^\bomega)\le M_h$. Thus $M_h-o_n(1)\le(\rho/\lambda_R)M_h$. Since $M_h>0$ is fixed and $\rho<\lambda_R/4$, this is impossible for all sufficiently large $n$. This proves the first assertion.

It remains to prove \eqref{eq:own-row-close}. Suppose that it fails for some $z\in V(X)$. By the first assertion, there is a vertex $z'\in V(Y)$ whose row is within $\rho$ of the row of $z$ in $X$. Let $t=\tau(z)$ and $t'=\tau(z')$. For $s\in V(R)$ define
\[
        D_s(z)\coloneqq
        \sum_{w\ne z}\bomega(w)\left|A_X(z,w)-A_R(s,\tau(w))\right|.
\]
Then $D_t(z)>\rho$, while the row closeness to $z'$ gives $D_{t'}(z)\le\rho$. If we change only the type of $z$ in $Y$ from $t$ to $t'$, the resulting graph $Y'$ is still a weighted blowup of $R^{\bnu'}$, where
\[
        \bnu'\coloneqq \bnu-\bomega(z)\mathbf 1_t+\bomega(z)\mathbf 1_{t'}\in\mathcal P(V(R)).
\]
The diagonal contribution is unchanged, and the weighted edit sum changes by
\[
        2\bomega(z)\bigl(D_{t'}(z)-D_t(z)\bigr)<0,
\]
so the new commensurable pair $(X^\bomega,Y'^\bomega)$ has strictly smaller weighted edit sum than $(X^\bomega,Y^\bomega)$. Since $X^\bomega\sim G^{u_G}$ and $Y'^\bomega$ is a weighted blowup of $R^{\bnu'}$, this new edit sum is an admissible upper bound for $\mathrm{edit}(G^{u_G},R^{\bnu'})$. This contradicts the choice of $\bnu$ minimizing $\mathrm{edit}(G^{u_G},R^\bnu)$ over $\mathcal P(V(R))$.
\end{proof}

For each $x\in V(G)$, independently choose a random point $Z_x\in S_x$ with law
\[
        \mathbb P(Z_x\in A)=n\bomega(A)\qquad\text{for all }A\subseteq S_x,
	\]
and put $L(x)\coloneqq\tau(Z_x)$. Define a graph $B$ on the vertex set $V(G)$ by taking
\[
        B_v\coloneqq L^{-1}(v)\qquad\text{for each }v\in V(R),
\]
and by joining every pair between $B_v$ and $B_w$ exactly when $vw\in R$, with no edges inside any $B_v$. Thus $B$ is a blowup of $R$. We now use \eqref{eq:own-row-close} to control the maximum error degree.

Fix $x\in V(G)$ and condition on $Z_x=z\in S_x$. For each $y\ne x$, let $I_y$ be the indicator of the event that $xy$ is an edge of $G\symdiff B$. After $Z_x$ is fixed, the variable $I_y$ depends only on the independent choice of $Z_y$, so the variables $\{I_y:y\ne x\}$ are independent Bernoulli variables, although their success probabilities may differ. Their conditional expected sum is
\[
        \mathbb E\Big[ \sum_{y\ne x}I_y\mid Z_x=z \Big]
        = n\sum_{y\ne x}\bomega\{w\in S_y:A_X(z,w)\ne A_Y(z,w)\}
        \le n\bomega\{w:A_X(z,w)\ne A_Y(z,w)\}
        \le \rho n,
\]
where the last inequality is \eqref{eq:own-row-close}. Since $\rho<\eta/20$, a standard Chernoff bound gives a constant $c_\eta>0$ such that, uniformly in $x$ and in $z$,
\[
        \mathbb P\bigl(d_{G\symdiff B}(x)>\eta n\mid Z_x=z\bigr)
        \le \mathrm{e}^{-c_\eta n}.
\]
Since the bound is uniform in the value of $Z_x$, averaging over $Z_x$ gives
\[
        \mathbb P\bigl(d_{G\symdiff B}(x)>\eta n\bigr)
        \le \mathrm{e}^{-c_\eta n}
        \qquad\text{for every }x\in V(G).
\]
Now
\[
        \{\Delta(G\symdiff B)>\eta n\}
        \subseteq
        \bigcup_{x\in V(G)}
        \{d_{G\symdiff B}(x)>\eta n\}.
\]
Therefore, by a union bound over the $n$ choices of $x$,
\[
        \mathbb P\bigl(\Delta(G\symdiff B)>\eta n\bigr)
        \le n\mathrm{e}^{-c_\eta n}=o(1).
\]

The same random labels also keep all parts large. Indeed,
\[
        \mathbb E[|B_v|]=n\bomega(\tau^{-1}(v))=\bnu(v)n\ge\lambda_R n
        \qquad\text{for all }v\in V(R),
\]
and the indicators $\mathbf 1_{\{L(x)=v\}}$, $x\in V(G)$, are independent. Applying Chernoff's bound to each fixed $v\in V(R)$ gives a constant $c_R>0$ such that
\[
        \mathbb P\bigl(|B_v|<\lambda_R n/2\bigr)
        \le \mathrm{e}^{-c_R n}
        \qquad\text{for every }v\in V(R).
\]
Since $V(R)$ is fixed, another union bound gives
\[
        \mathbb P\bigl(\exists v\in V(R): |B_v|<\lambda_R n/2\bigr)
        \le |V(R)|\mathrm{e}^{-c_R n}=o(1).
\]
Together with the preceding maximum-degree estimate, this implies
\[
        \mathbb P\bigl(\Delta(G\symdiff B)\le\eta n
        \text{ and } |B_v|\ge\lambda_R n/2\text{ for every }v\in V(R)\bigr)
        =1-o(1).
\]
In particular, this probability is positive for all sufficiently large $n$. For such an outcome, $B$ satisfies $|B_v|\ge\lambda_R n/2=\lambda_0 n$ for every $v$ and $\Delta(G\symdiff B)\le\eta n$, proving the proposition.
\end{proof}

\subsection{Cleaning wrong pairs}
\label{sec:cleaning}

We now show that the error graph supplied by \cref{prop:finite-stability},
\[
        J\coloneqq G\symdiff B
        \qquad\text{with}\qquad
        \Delta(J)\le\eta n,
\]
must in fact be empty. Concretely, take an embedding $\psi\in\Emb(F,G)\setminus\Emb(F,B)$, and let $\theta$ send each vertex $a\in V(F)$ to the part of $B$ containing $\psi(a)$. Since $\psi$ is not an embedding into $B$, there are distinct vertices $a,b\in V(F)$ such that
\[
        A_F(a,b)\ne A_R(\theta(a),\theta(b)).
\]
Thus $\theta$ is not a strong homomorphism from $F$ to $R$. The next lemma says that, for a large blowup $F=H^{(h)}$, this failure is witnessed around one vertex by linearly many adjacency disagreements. These disagreements will later force a large star in the error graph.

\begin{lemma}
\label{lem:many-forced-errors}
Let $H=R(\mathbf k)$, where $R$ is twin-free and $|V(R)|\ge2$. There is a constant $c_*=c_*(H)>0$ such that the following holds for every sufficiently large $h$. Put $F\coloneqq H^{(h)}=R(h\mathbf k)$. If $\theta:V(F)\to V(R)$ is not a strong homomorphism, then there are $a_0\in V(F)$ and $U\subseteq V(F)\setminus\{a_0\}$ with $|U|\ge c_*h$ such that
\[
        A_F(a_0,a)\ne A_R(\theta(a_0),\theta(a))
        \qquad\text{for all }a\in U.
\]
\end{lemma}

\begin{proof}
Let $r\coloneqq |V(R)|$, and let $k_{\min}\coloneqq\min_x\mathbf k(x)$. For each $x\in V(R)$, choose a value $\sigma(x)\in V(R)$ which occurs most often among the values of $\theta$ on the type class $\pi^{-1}(x)$, and set
\[
        M_x\coloneqq\{a\in\pi^{-1}(x):\theta(a)=\sigma(x)\}.
\]
Then
\[
        |M_x|\ge \frac{h\mathbf k(x)}{r}\ge \frac{hk_{\min}}{r}
        \qquad\text{for every }x\in V(R).
\]

First suppose that $\sigma:R\to R$ is not a strong homomorphism. Choose $x,y\in V(R)$ such that
\[
        A_R(x,y)\ne A_R(\sigma(x),\sigma(y)).
\]
Here $x\ne y$, since both graphs are loopless. Pick any $a_0\in M_x$ and take $U=M_y$. For every $a\in U$,
\[
        A_F(a_0,a)=A_R(x,y)
        \ne A_R(\sigma(x),\sigma(y))
        =A_R(\theta(a_0),\theta(a)),
\]
and $|U|\ge hk_{\min}/r$.

It remains to consider the case where $\sigma$ is a strong homomorphism. Since $R$ is twin-free, \cref{lem:twin-free-auto} implies that $\sigma$ is an automorphism. If $\theta(a)=\sigma(\pi(a))$ for every $a\in V(F)$, then $\theta$ would be a strong homomorphism, contrary to the assumption. Hence there is $a_0\in V(F)$ such that $\theta(a_0)\ne\sigma(\pi(a_0))$. By twin-freeness, there is $y\in V(R)$ such that
\[
        A_R(\theta(a_0),y)\ne A_R(\sigma(\pi(a_0)),y).
\]
Since $\sigma$ is onto, we may write $y=\sigma(w)$ for some $w\in V(R)$. Thus
\[
        A_R(\theta(a_0),\sigma(w))
        \ne A_R(\sigma(\pi(a_0)),\sigma(w))=A_R(\pi(a_0),w).
\]
Now take $U\coloneqq M_w\setminus\{a_0\}$. For every $a\in U$,
\[
        A_F(a_0,a)=A_R(\pi(a_0),w)
        \ne A_R(\theta(a_0),\sigma(w))
        =A_R(\theta(a_0),\theta(a)).
\]
Moreover, $|U|\ge hk_{\min}/r-1$.

Thus, in both cases, the desired conclusion holds with $c_*\coloneqq k_{\min}/(2r)$, after increasing the lower threshold on $h$ if necessary.
\end{proof}

The following lemma is the exact cleaning step. It compares embeddings lost and created by a small maximum-degree error graph and shows that any non-empty error graph decreases the number of induced embeddings.

\begin{lemma}
\label{lem:cleaning}
Let $H=R(\mathbf k)$, where $R$ is twin-free and $|V(R)|\ge2$. For every $\lambda>0$ there are $\eta=\eta(H,\lambda)>0$ and $h_{\ref{lem:cleaning}}=h_{\ref{lem:cleaning}}(H,\lambda)$ such that the following holds. Let $h\ge h_{\ref{lem:cleaning}}$ and put $F\coloneqq H^{(h)}$. If $n$ is sufficiently large, $B$ is a blowup of $R$ on a vertex set $V$ with every part of size at least $\lambda n$, and $G$ is another graph on $V$ such that the graph $J\coloneqq G\symdiff B$ satisfies $|J| > 0$ and $\Delta(J)\le\eta n$, then $N_F(B)>N_F(G)$.
\end{lemma}

\begin{proof}
Let $c_*$ be the constant from \cref{lem:many-forced-errors}. We first choose the auxiliary constants used in the counting argument.
Choose $h_{\ref{lem:cleaning}}$ large enough that the following holds for every $h\ge h_{\ref{lem:cleaning}}$. Put
\[
        F\coloneqq H^{(h)},\qquad
        m\coloneqq |V(F)|=h|V(H)|,\qquad
        \ell\coloneqq\lfloor c_*h\rfloor.
\]
We can choose $h_{\ref{lem:cleaning}}$ sufficiently large so that, for this graph $F$, the conclusion of \cref{lem:many-forced-errors} is available for every map $\theta:V(F)\to V(R)$ that is not a strong homomorphism. We also require that every type class of $F$ has size at least two, i.e., $h\mathbf k(t)\ge2$ for every $t\in V(R)$, and that the following numerical inequalities hold:
\[
        m2^m\le3^m,
        \qquad
        \ell-1\ge c_*h/3,
        \qquad
        2\cdot 3^m\eta^{\ell-1}
        < \frac{1}{\binom m2}\left(\frac{\lambda}{2}\right)^{m-2}.
\]
The last requirement is possible because $m=h|V(H)|$ and we may choose $\eta$ small enough so that
\[
        \frac{3^{|V(H)|}\eta^{c_*/3}}
             {(\lambda/2)^{|V(H)|}}
             < \frac{1}{2},
\]
hence the exponential decay in $h$ absorbs the polynomial factor $\binom m2$.

Now fix $h\ge h_{\ref{lem:cleaning}}$, and use the notation $F,m,\ell$ above. Let $B$ and $G$ be as in the statement of~\cref{lem:cleaning}, set $J\coloneqq G\symdiff B$, and put $K\coloneqq |J|$.
We compare the embeddings lost when passing from $B$ to $G$ with the embeddings created by this change. Define
\[
        L\coloneqq |\Emb(F,B)\setminus\Emb(F,G)|,
        \qquad
        P\coloneqq |\Emb(F,G)\setminus\Emb(F,B)|.
\]
It is enough to prove $L>P$, since then $|\Emb(F,B)|>|\Emb(F,G)|$, equivalently $N_F(B)>N_F(G)$.

\begin{claim}
\label{claim:cleaning-loss-lower-bound}
For all sufficiently large $n$, we have $L\ge \frac{K}{\binom m2}\left(\frac{\lambda}{2}\right)^{m-2}n^{m-2}$.
\end{claim}

\begin{proof}
Fix a wrong pair $xy\in J$. Choose an arbitrary ordering of this unordered pair, and write $x\in B_u$ and $y\in B_v$. If $u\ne v$, choose $a,b\in V(F)$ with types $u,v$, respectively; if $u=v$, choose two distinct vertices $a,b$ of type $u$. Then
\[
        A_F(a,b)=A_B(x,y)\ne A_G(x,y).
\]
Every type-respecting embedding of $F$ into $B$ with $a\mapsto x$ and $b\mapsto y$ is therefore destroyed in $G$. If
\[
        r_t\coloneqq |\{a,b\}\cap \pi^{-1}(t)|\qquad\text{for all }t\in V(R),
\]
then $\sum_t r_t=2$. After the prescribed images $x,y$ have been used, the remaining $h\mathbf k(t)-r_t$ vertices of type $t$ in $F$ must be injected into the unused vertices of $B_t$. Therefore the number of type-respecting embeddings with $a\mapsto x$ and $b\mapsto y$ is
\[
        \prod_{t\in V(R)} (|B_t|-r_t)_{h\mathbf k(t)-r_t},
\]
which is at least
\[
        \prod_{t\in V(R)}\left(\frac{\lambda n}{2}\right)^{h\mathbf k(t)-r_t}
        = \left(\frac{\lambda n}{2}\right)^{m-2},
\]
for fixed $h$ and all sufficiently large $n$. Summing this bound over all $K$ wrong pairs and then dividing by $\binom m2$ gives the claim, because the image of any embedding contains at most $\binom m2$ unordered vertex pairs whose adjacency differs between $B$ and $G$.
\end{proof}

\begin{claim}
\label{claim:cleaning-gain-upper-bound}
For all sufficiently large $n$, we have $P\le 2\cdot 3^m\eta^{\ell-1}K n^{m-2}$.
\end{claim}

\begin{proof}
Let $\psi\in\Emb(F,G)\setminus\Emb(F,B)$. Let $\pi_B:V(B)\to V(R)$ be the type map of $B$, and set $\theta\coloneqq\pi_B\circ\psi$. Since $\psi$ is not an embedding into $B$, the map $\theta$ is not a strong homomorphism. By \cref{lem:many-forced-errors}, the image of $\psi$ contains a star with at least $\ell$ leaves in the error graph $J$.

Therefore $P$ is at most the number of ways to specify $\ell$ leaves of such a forced star and then place the remaining vertices. Indeed, the center preimage can be chosen in at most $m$ ways, the $\ell$ leaf preimages in at most $\binom{m-1}{\ell}\le2^m$ ways, the center image as some $x\in V$, the leaf images in at most $d_J(x)^\ell$ ways, and the remaining vertices in at most $n^{m-\ell-1}$ ways. Hence
\[
        P\le m2^m\sum_{x\in V} d_J(x)^\ell n^{m-\ell-1}
        \le 3^m\sum_{x\in V} d_J(x)^\ell n^{m-\ell-1}.
\]
Using $d_J(x)\le\eta n$ and $\sum_xd_J(x)=2K$, we get $P\le 2\cdot 3^m\eta^{\ell-1}K n^{m-2}$, as claimed.
\end{proof}

Recall that by the choice of $h_{\ref{lem:cleaning}}$, we have $2\cdot 3^m\eta^{\ell-1} < \frac{1}{\binom m2}\left(\frac{\lambda}{2}\right)^{m-2}$.
Combining \cref{claim:cleaning-loss-lower-bound,claim:cleaning-gain-upper-bound}, for all sufficiently large $n$ we have
\[
        P\le 2\cdot 3^m\eta^{\ell-1}K n^{m-2}
        < \frac{K}{\binom m2}\left(\frac{\lambda}{2}\right)^{m-2}n^{m-2}
        \le L.
\]
Thus $N_F(B)>N_F(G)$.
\end{proof}

\subsection{Proof of the main theorem}
\label{sec:main-proof}

Write $H=R(\mathbf k)$, where $R\coloneqq\widetilde H$ is twin-free. If $|V(R)|=1$, then $H=\overline K_q$ for some $q\ge1$. Choose $h_*=2$ when $q=1$, and $h_*=1$ otherwise. Then $qh\ge2$ for every $h\ge h_*$. For such $h$ and all $n\ge qh$, the unique graph maximizing the number of induced copies of $\overline K_{qh}$ is the edgeless graph: if a graph has an edge, every $qh$-set containing that edge is not independent, while the edgeless graph makes all $qh$-sets independent. Since the edgeless graph is a blowup of $H$, the theorem holds in this case. Hence assume $|V(R)|\ge2$.

Let $\lambda_0$ be given by \cref{prop:finite-stability}. Apply \cref{lem:cleaning} with $\lambda=\lambda_0$, obtaining $\eta>0$ and $h_{\ref{lem:cleaning}}$. Apply \cref{prop:finite-stability} with this $\eta$, obtaining $h_{\ref{prop:finite-stability}}$. Set
\[
        h_*\coloneqq\max\{h_{\ref{lem:cleaning}},h_{\ref{prop:finite-stability}}\}.
\]

Fix $h\ge h_*$, put $F\coloneqq H^{(h)}$, and let $G$ be an $n$-vertex graph maximizing $N_F(G)$, where $n$ is sufficiently large. By \cref{prop:finite-stability}, there is a blowup $B$ of $R$ on $V(G)$ such that every part has size at least $\lambda_0 n$ and $\Delta(G\symdiff B)\le\eta n$. If $G\ne B$, then \cref{lem:cleaning} gives $N_F(B)>N_F(G)$, contradicting the extremality of $G$. Therefore $G=B$, so $G$ is a blowup of $R$.

Finally, each part of $B$ has linear size. For all sufficiently large $n$, the part corresponding to $v\in V(R)$ can be split into $\mathbf k(v)$ non-empty subparts. All vertices inside one quotient part of $B$ have identical adjacencies outside the part and no adjacencies inside the part, so this refinement makes $G$ a blowup of $H=R(\mathbf k)$. The theorem follows.

\section*{Acknowledgments}

\begin{sloppypar}
X.L. was supported by the Excellent Young Talents Program (Overseas) of the National Natural Science Foundation of China.
\end{sloppypar}

\bibliographystyle{abbrv}
\bibliography{Inducibility}

\end{document}